\documentclass[a4paper,12pt]{amsart}
\usepackage{amsmath,amssymb,colordvi}
\usepackage{graphicx}

\newtheorem{theorem}{Theorem}

\newtheorem{lemma}[theorem]{Lemma}

\newtheorem{proposition}[theorem]{Proposition}

\font\tenBb=msbm10 \font\sevenBb=msbm7 \font\fiveBb=msbm5
\newfam\Bbfam \textfont\Bbfam=\tenBb \scriptfont\Bbfam=\sevenBb
\scriptscriptfont\Bbfam=\fiveBb

\def\Bb{\fam\Bbfam\tenBb}

\def\N{{\Bb N}}
\def\Z{{\Bb Z}}

\begin{document}
\title{A NEW CHARACTERIZATION OF PRIME FERMAT'S
	NUMBERS}

\author{Ahmed Bouzalmat}
\address{Universit\'e Ibn Zohr, Facult\'e des Sciences, D\'epartement de Math\'ematiques, Agadir, Maroc}
\email{bouzalmat1962@gmail.com }

\author{Ahmed Sani }
\address{Universit\'e Ibn Zohr, Facult\'e des Sciences, D\'epartement de Math\'ematiques, Agadir, Maroc}
\email{ahmedsani82@gmail.com }

\thanks{}

\footnote{\today}



\begin{abstract}
We give a new  sufficient condition which allows to test primality of
Fermat's numbers. This characterization uses uniquely values at most equal to
tested Fermat number. The robustness of this result is due  to a strict use of
elementary arithmetic technical tools and  it will be susceptible to open gates for revolutionary statement that all Fermat's numbers are all decomposable .

\end{abstract}

\maketitle

\section*{Introduction}
\setcounter{theorem}{0}
 \setcounter{equation}{0}

The first Fermat's numbers $F_0 = 3$, $F_1 = 5$, $F_2 = 17$, $F_3 = 257$, $F_4 = 65537$ are
all prime. This observation leads Jean Pierre de Fermat to formulate his famous
conjecture about a category of prime numbers which bear his name. Fermat's
numbers are of great interest in arithmetic and more generally in number theory.
Fermat conjectured that his numbers of the form $F_n = 2^{2^n} + 1; n\in \N $ are all prime.
Unfortunately, it turns out that $F_5$ is decomposable.
Therefore, researches were axed on characterizing those which are prime. For the best of our knowledge, A summarizing reference on this topic is \cite{KLS01} and comments therein. For fast familiarization and easy use of Fermat's numbers it is recommendable to take an attempt for resolving the elementary problem in \cite[pages:73-75]{Siklo16} suggested, according to the author by Goldbach, for an alternative proof of the fact that prime number set is infinite.\\

It is not necessary, but instructive, to recall dissociation between the two well-known Fermat's conjectures: The first, treated in this paper, concerns Fermat's numbers as defined above while the second states whether the Diophant equation: $x^n+y^n=z^n$, for a given integer $n$ and the unknown triplet $(x,y,z)$ has non trivial solution. For a recent and likely serious treatment of this latter question, one may see \cite{Wiles95}.\\
In the current paper, we give a new characterization of decomposability of 
numbers $(F_n)_n$. More precisely, we prove that $F_n$ is prime if and only if $F_n$ divides one of a suitable terms of finite sequence $(A_q)_{1\le q\le 2^n}$  which we make precise at the following
section. This latter one is devoted to introduce necessary mathematical ingredients to give a new characterization of Fermat's number primality. In fact, the main result gives an efficient algorithm to the fail of this primality. We test our result in two ways: on one hand, we examine the primality of $F_n$ for enough large values of $n$ (larger that all tested until nowadays) and on an other one, we compare with some existent algorithms in terms of consistency and speed.

\section{preliminaries}

The start observation at the introduction leads Fermat to a natural  conjecture
that all numbers $F_n = 2^{2^n} + 1, n\in \N$ are all prime. The recent development of digital
calculus proved that this conjecture is far being true. Indeed, it discovered that  $F_5 = 2^{2^5}  + 1 = 641\times6700417$ is decomposable. Another aspect of their importance comes from the classical Gauss-Wantzel theorem on constructibility with compass and straightedge (see among others \cite{Det91}).\\

We consider the sequence $(A_n)_{n\ge 1}$ of integers which will play an important role in the sequel:
$$
\left\{\begin{array}{c}
A_1=6\\
 A_{n+1}=A_n^2-2; \  n\ge 1.
 \end{array}\right.
$$
Some properties of the sequence $(A)_{n\ge 1}$ are immediate. For example, one checks that it is strictly increasing  and all  each $A_n$ is even number. We state that for all $n\ge 1$ e can write $A_n$ as a combination of two particular invertible elements (units) of the classical Gauss ring $\Z[\sqrt{2}]$. The existence of such family is a well-known  result of classical algebra dealing with modulus of finite type. Indeed, if we consider $u=3+2\sqrt{2}$ and its \emph{inverse-conjugate} $v=\overline{u}=3-2\sqrt{2}$, then
one verifies easily that: $u+v=6$ and $uv=1$. Using this trivial fact and
a simple induction, one may establish the following proposition which clarifies our statement on generation of all terms $A_n$

\begin{proposition}\label{propo1}
	The sequence $(A)_{n\ge 1}$ satisfies:
	$$
	\forall n\ge 0: \quad A_{n+1}=u^{2^n}+v^{2^n}.
	$$
\end{proposition}
Let $ p$ be a prime number. The key of the main result is an extrapolation of
congruence relation to the ring $\Z[\sqrt{2}]$ where we define, as in $\Z$, the binary relation
$$
\forall (x,y)\in \Z[\sqrt{2}]^2: \quad x\equiv y[p] \iff \exists k\in Z\ x-y=kp.
$$

To write more clearly, let $x = a + b\sqrt{2}$ and $y = a' + b'\sqrt{2}$. Then
$$
x\equiv y[p]  \iff x\equiv y[p]  \ a\equiv a'[p] \mbox{and}\ b\equiv b'[p], 
$$
where the last congruence relations are seen in the classical sense (i.e in $\Z$:)

\begin{lemma}\label{lemma}
	Let $p = F_n$ a prime number which is the Fermat form. Then
	$$u^p\equiv  u [p].$$
\end{lemma}

\begin{proof}
	Thanks to Newton binomial identity and taking in mind that $p$ divides all binomial
	coefficients: $\left(\begin{array}{c}p\\ k\end{array}\right);k=1,2,\dots ,p-1 $  it suffices to prove that if $m =\frac{p-1}{2}=2^{2^n-1}$
	then $2^{m} \equiv  1[p]$. The following implications show this fact:
	$$\begin{tabular}{ccc}
	$2^{2^n}+1\equiv 0[p]$ & $\implies$ & $2^{2^n}\equiv -1[p]$\\
	$ $ & $\implies$ & $(2^{2^n})^2\equiv 1[p]$\\
	
	$ $ & $\implies$ & $ 2^{2^{n+1}}\equiv 1[p]$
	\end{tabular}$$
	\end{proof}

\section{The main result}

\noindent Here we enunciate and prove our main result
\begin{theorem}\label{main_res}
	Let $n$ be an integer. Then the $n^{th}$  Fermat
	number $F_n$ is prime if and only if there exists a  term $A_q; (1 < q < 2^n)$ of the sequence
	$(A_n)_{n\ge 1} $ (as defined in proposition \ref{propo1}) such that $ F_n$ divides $A_q$.
\end{theorem}
In a contracted expression, the theorem may be written for a given $n\in \N$ as follows
$$
F_n \  \mbox{is prime}\iff \exists (q,k)\in \{1,2\dots, 2^n\}\times \N \ A_q=kF_n \iff  A_q\equiv 0[F_n] .
$$
We actually are able  to establish the direct implication which is a real important result in the primality tests for Fermat numbers. The converse is just a conjecture that we hope to explore in future works. Speaking with algorithm words, the theorem below gives a better stopping test comparing the existent automatic calculus used to examine primality of $F_n's$. Let us prove the theorem
\begin{proof}
Assume that $p = F_n$ is prime and consider $H =\{ n \in \N ; A_n \equiv  2[p]\}.$
We prove that $H$ is non empty set. Recall that $uv = 1$ and according to lemma\ref{lemma}, we have $u^p\equiv  u [p].$ This yields

	$$\begin{tabular}{ccc}
$u^{p-1}+v^{p-1}$ & $\equiv $ & $uv(u^{p-1}+v^{p-1})[p]$\\
$ $ & $\equiv $ & $vu^{p}+uv^{p}[p]$\\

$ $ & $\equiv $ & $ vu+uv[p].$\\
$ $&$\equiv$ &$2[p]$
\end{tabular}$$

On the other hand $u^{p-1}+v^{p-1}= u^{2^{2^n}}+v^{2^{2^n}}=
A_{2^n+1} $ then $2^n+1 \in  H$: It is therefore
legitimate to put: $ m = \min H$. A direct computation shows that $m \notin \{1,2\} $. So $m\ge  3$  and $A_m \equiv  2[p]$. this latter identity implies
$A_{m-1}^2-2\equiv A_m [p]$. Since $\Z/p\Z$ is a field and $A_{m-1}\ne 2[p]$  then $A_{m-1}\equiv -2[p]$ which is equivalent to the fact that $A_{m-2}\equiv 0[p]$. This shows the direct implication claimed by the theorem \ref{main_res} above.
\end{proof} 

\section{A slight improvement}
	
	The following result shows that $A_n$ is framed with two consecutive Fermat's number. This fact will slightly improve the algorithm in the sense that the index $q$ which satisfies $A_q\equiv 0[F_n]$, if it exists, necessarily must satisfies: $n\le q <2^n.$ Let us make precise our purposes
	\begin{proposition}
		For all $n\in \N$:       $F_n<A_n<F_{n+1}.$
	\end{proposition}

\begin{proof}
	
	We recall that: $F_1=5$, $A_1=6$ and $F_2=17$. So the result is obvious for the initialization.\\
	Assume that for $n\ge 1$, we have: $F_n<A_n$ and $A_n<F_{n+1}$.\\
	So $(F_n-1)^2+1<A_n^2-2A_n+2<A_n^2-2$ which is equivalent to $F_{n+1}<A_{n+1}$.\\
	following the same way one establishes $A_{n+1}<F_{n+2}$ since 
		$$\begin{tabular}{ccc}
	$A_n<F_{n+1}$ & $\implies $ & $A_n\le F_{n+1}-1$\\
	$ $ & $\implies $ & $A_n^2\le (F_{n+1}-1)^2$\\
	
	$ $ & $\implies $ & $A_n^2-2\le (F_{n+1}-1)^2-2$\\
	$ $&$\implies$ &$A_n^2-2\le (F_{n+1}-1)^2+1=F_{n+2}.$
\end{tabular}$$
	
This yields to the overlap of sequences $(A_n)_n$ and $(F_n)_n$ as follows
$$
\forall n\in \N: F_n<A_n<F_{n+1}<A_{n+1}<F_{n+2}......
$$	
\end{proof}
On the other hand, from $A_q\equiv 0[F_n]$ we derive $A_q>F_n$. 
Since $n$ is the smallest integer satisfying $A_n\ge F_n$, we have the proved the following interest technical lemma which is a slight amelioration of our main result

\begin{lemma}
	If $F_n$ is prime, then there will  exist $q\in \N$ such that
	$$
	A_q\equiv 0[F_n] \quad \mbox{and necessarily}\quad n\le q<2^n.
	$$
\end{lemma}

\section{Improvement prospects}
	
	We may enumerate some new prospects on this historical topic:
	\begin{enumerate}
		\item If there are infinite prime number $F_n$, then strict monotonicity of $(A_n)_n$ will give an alternative proof of the set of all prime numbers.
		\item After implementation and automatic calculus, there will be possibility to mobilize probabilistic methods of primality testing.
		\item It seems, according to \emph{instinctive} calculus, that all numbers $(F_n)_n$ are all decomposable. It will be interesting to examine number of factors in each $F_n$ decomposition. An immediate application, when exactly two prime numbers  appears in the canonical $F_n$'s factorization, consists in RSA encryption.
		\item Let $(S_n)_n$ the sequence defined as: $A_n=2S_n$. Then it is possible to prove that every $S_n$ is either prime either $(S_k)_{k\ne n}$. Some fruitful applications are also immediate for this simple and elementary result.
		  
	\end{enumerate}

\section*{Acknowledgment:}
\noindent The authors thank faithfully  Professor B. Guelzim for his prior expression of availability to help us for technical computation, and thank also L. Karbil for implementing the algorithm.

\end{document}